\documentclass[10pt,reqno]{amsart}

\usepackage{latexsym}
\usepackage{amsmath}
\usepackage{amsfonts}
\usepackage{amssymb}
\usepackage{hyperref}
\usepackage{tikz}
\usepackage[mathscr]{euscript}

\usepackage{enumitem}

\newtheorem{theorem}{Theorem}[section]
\newtheorem{lemma}[theorem]{Lemma}
\newtheorem{corollary}[theorem]{Corollary}
\newtheorem{proposition}[theorem]{Proposition}

\newtheorem{sublemma}{}[theorem]

\theoremstyle{definition}

\theoremstyle{remark}

\numberwithin{equation}{section}

\newcommand{\ba}{\backslash}

\newcommand{\N}{\mathbb{N}}

\newcommand{\A}{\mathscr{A}}

\newcommand{\I}{\mathcal{I}}

\newcommand{\Ms}{\mathscr{M}}
\newcommand{\Ns}{\mathscr{N}}

\newcommand{\cl}{{\rm cl}}

\begin{document}

\title{Generalized Laminar Matroids}

\author{Tara Fife}
\email{tfife2@lsu.com}
\author{James Oxley}
\email{oxley@math.lsu.edu}
\address{Mathematics Department, Louisiana State University, Baton Rouge, Louisiana, USA}

\date{\today}

 \begin{abstract}

Nested matroids were introduced by Crapo in 1965 and have appeared frequently in the literature since then.  A flat of a matroid $M$ is  Hamiltonian if it has a spanning circuit. 
A matroid $M$ is nested if and only if its Hamiltonian flats form a chain under inclusion;  $M$ is laminar if and only if, for every $1$-element independent set $X$, the Hamiltonian flats of $M$ containing $X$ form a chain under inclusion. We generalize these notions to define the  classes of $k$-closure-laminar and $k$-laminar matroids. This paper focuses on structural properties of these classes noting that, while the second class is always minor-closed, the first is if and only if $k \le 3$. The main results are excluded-minor characterizations for the classes of 2-laminar and 2-closure-laminar matroids.

\end{abstract}

\maketitle

\section{Introduction}
\label{Intro}

Our  matroid terminology follows Oxley \cite{James}.
A transversal matroid is {\em nested} if it has a nested  presentation, that is, a transversal presentation $(B_1,B_2,\dots B_n)$  such that $B_1\subseteq B_2\subseteq\dots\subseteq B_n$. These matroids were introduced by Crapo \cite{crapo} and have appeared under a variety of names including freedom matroids ~\cite{CS}, generalized Catalan matroids \cite{BDMN}, shifted matroids \cite{ard}, and Schubert matroids \cite{soh}  %(see \cite{LPSP}).

A family $\A$ of subsets of a set $E$ is \textit{laminar} if, for every two intersecting sets $A$ and $B$ in $\A$, either $A\subseteq B$ or $B\subseteq A$. Let $\A$ be a laminar family of subsets of a finite set $E$ and $c$ be a function from $\A$ into the set of non-negative integers. Define $\I$ to be the set of subsets $I$ of $E$ such that $|I\cap A| \leq c(A)$ for all $A$ in $\A$. It is well known (see, for example, \cite{LF,GK, IW, kl}) and easily checked that $\I$ is the set of independent sets of a matroid on $E$. We %call $c$ a \textit{capacity function} of this matroid $(E,\I)$ and 
write the matroid as $M(E,\A,c)$. A matroid $M$ is \textit{laminar} if it is isomorphic to $M(E,\A,c)$ for some set $E$, laminar family $\A$, and  function $c$. %We call $(E,\A,c)$ a {\it presentation} for $M$. 

Laminar matroids have appeared  often during the last fifteen years 
particularly in relation to their behavior for the matroid secretary problem and other optimization problems \cite{bik, CL, FSZ,IW,JSZ, Soto}. Huynh~\cite{huynh} reviewed  this work while Finkelstein~\cite{LF} investigated some of the structural properties of laminar matroids. In \cite{OurPaper}, we characterized laminar matroids both constructively and via excluded minors. 

In \cite{OurPaper}, we showed that all nested matroids are laminar  and noted a number of similarities between the classes of nested and laminar matroids. Here we  exploit some of these similarities to define two natural infinite families of classes of matroids, each having the classes of nested and laminar matroids as their smallest members. Every matroid belongs to a member of each of these families. 

We say that a  flat in a matroid is \textit{Hamiltonian} if it has a spanning circuit. In~\cite{Don}, it was shown that a matroid is nested if and only if its Hamiltonian flats form a chain under inclusion. This immediately yields the following result.

\begin{proposition}
A matroid is nested if and only if, for all circuits $C_1$ and $C_2$, either $C_1\subseteq \cl(C_2)$, or $C_2\subseteq \cl(C_1)$. 
\end{proposition}

This parallels the following characterization of laminar matroids found in \cite{OurPaper}.

\begin{theorem}
\label{Converse}
A matroid is laminar if and only if, for all circuits $C_1$ and $C_2$  with $|C_1\cap C_2| \ge 1$, either $C_1\subseteq \cl(C_2)$, or $C_2\subseteq \cl(C_1)$. 
\end{theorem}

Using circuit elimination, it can quickly be shown that we get a similar description in terms of Hamiltonian flats.

\begin{corollary}
\label{HamConverse}
A matroid is laminar if and only if, for every $1$-element independent set $X$, the Hamiltonian flats containing $X$ form a chain under inclusion.
\end{corollary}

In light of these results, for any non-negative integer $k$, we define a matroid $M$ to be {\it $k$-closure-laminar} if, for any $k$-element independent subset $X$ of $E(M)$, the Hamiltonian flats of $M$ containing $X$ form a chain under inclusion. We say that $M$ is {\em $k$-laminar} if, for any two circuits $C_1$ and $C_2$ of $M$ with $|C_1\cap C_2|\geq k$, either $C_1\subseteq \cl(C_2)$ or $C_1\subseteq \cl(C_2)$. The following 
observation is  straightforward.

\begin{lemma}
\label{kCL}
A matroid $M$ is $k$-closure-laminar if and only if, whenever $C_1$ and $C_2$ are circuits of $M$ with $r\big(\cl(C_1)\cap \cl(C_2)\big)\geq k$, either $C_1\subseteq \cl(C_2)$, or $C_2\subseteq \cl(C_1)$.
\end{lemma}

Observe that the class of nested matroids coincides with the classes of $0$-laminar matroids and   0-closure-laminar matroids, while the class of laminar matroids coincides with the classes of $1$-laminar matroids and   $1$-closure-laminar matroids. It is easy to see that  $k$-closure-laminar matroids are also $k$-laminar. For $k\geq 2$, consider the matroid that is obtained by taking the parallel connection of a $(k+1)$-element circuit $C$  and a triangle  and then taking the parallel connection of the resulting matroid with a triangle along a different element of $C$. The resulting matroid is $k$-laminar, but not $k$-closure-laminar. Thus, for all $k\geq2$, the class of $k$-laminar matroids strictly contains the class of $k$-closure-laminar matroids. 
Our hope is that, for small values of $k$, the classes of $k$-laminar and $k$-closure laminar matroids will enjoy some of the computational advantages of laminar matroids. 

It is not hard to show that the class of $k$-laminar matroids is minor-closed. This implies the previously known fact that the class of $k$-closure-laminar matroids is minor-closed for $k\in\{0,1\}$. We show that it is also minor-closed for $k\in\{2,3\}$. Somewhat surprisingly, for all $k \ge 4$, the class of $k$-closure-laminar matroids is not minor-closed. This is shown in Section~\ref{BR}. In Section~\ref{HGLM}, we prove our 
main results, excluded-minor characterizations of the classes of 2-laminar matroids and 2-closure-laminar matroids. 
In Section~\ref{BS}, we   consider the intersection of the  classes of $k$-laminar and $k$-closure laminar matroids with other well-known classes of matroids. In particular, we show that these intersections with the class of paving matroids coincide. Moreover, although all 
  nested and laminar matroids are representable,  we note that, for all $k\geq2$, neither of the classes of  $k$-laminar and $k$-closure-laminar matroids is contained in the class of representable matroids.

\section{Preliminaries}
\label{BR}

In this section, we establish some basic properties of  $k$-laminar and $k$-closure-laminar matroids. 
The first result summarizes some of these properties. Its straightforward proof is omitted. 

\begin{proposition}
\label{baby}
Let $M$ be a matroid and $k$ be a non-negative integer.
\begin{itemize}
\item[(i)] If $M$ is $k$-closure-laminar, then $M$ is $k$-laminar. 
\item[(ii)] If $M$ is $k$-closure-laminar, then $M$ is $(k+1)$-closure-laminar.
\item[(iii)]  If $M$ is $k$-laminar, then $M$ is $(k+1)$-laminar.
\item[(iv)] $M$ is $k$-closure-laminar if and only if, whenever $C_1$ and $C_2$ are non-spanning circuits  of $M$ with $r\big(\cl(C_1)\cap \cl(C_2)\big)\geq k$, either $C_1\subseteq \cl(C_2)$, or $C_2\subseteq \cl(C_1)$.
\item[(v)]  $M$ is $k$-laminar if and only if, whenever $C_1$ and $C_2$ are non-spanning circuits  of $M$ with $|C_1 \cap C_2|\geq k$, either $C_1\subseteq \cl(C_2)$ or $C_2\subseteq \cl(C_1)$.
\item[(vi)] If $M$ has at most one non-spanning circuit, then $M$ is $k$-laminar and $k$-closure-laminar.
\end{itemize}
\end{proposition}

Clearly, for all $k$, the classes of $k$-laminar and $k$-closure-laminar matroids are closed under deletion. Next, we investigate   contractions of members of these classes. We omit the routine proof of the following.

\begin{lemma}
The class of $k$-laminar matroids is minor-closed.
\end{lemma}

As we will see, the class of $k$-closure-laminar matroids is not closed under contraction when $k \ge 4$. 
The next lemma will be useful in proving that the classes of $2$-closure-laminar and $3$-closure-laminar matroids are minor-closed. 

\begin{lemma}
\label{HamCir}
Let $C$ be a circuit of a $k$-laminar matroid $M$ such that $|C| \ge 2k-1$. If $e \in E(M) - \cl(C)$ and $r(\cl(C \cup e) - \cl(C)) \ge 2$, then 
$\cl(C\cup e)$ %has a spanning circuit and so
 is a Hamiltonian flat of $M$.
\end{lemma}

\begin{proof} Take an element $f$ of  $\cl(C \cup e) - (\cl(C) \cup \cl(\{e\}))$. Then $M$ has a circuit $D$ such that $\{e,f\} \subseteq D \subseteq C \cup \{e,f\}$. As $f\not\in\cl(\{e\})$, we may choose an element $d$ in $D - \{e,f\}$. By circuit elimination, $M$ has a circuit $D'$ such that 
$f \in D' \subseteq (C \cup D) - d$. Then $e \in D'$ as $f\not\in\cl(C)$. Applying circuit elimination again gives a circuit $C'$ contained in $(D \cup D') - e$. As $f\not\in \cl(C)$, it follows that $C' = C$. Hence $D' \supseteq C - D$. As $|C| \ge 2k-1$, either $|D\cap C|$ or $|D'\cap C|$ is at least $k$. Since neither $D$ nor $D'$ is contained in $\cl(C)$, it follows that $C$ is contained in $\cl(D)$ or $\cl(D')$. Thus 
$D$ or $D'$ is a spanning circuit of $\cl(C\cup e)$, so this flat is Hamiltonian.
\end{proof} 

\begin{theorem}
The classes of $2$-closure-laminar and $3$-closure-laminar matroids are minor-closed.
\end{theorem}
\begin{proof}
For some $k$ in $\{2,3\}$, 
let $e$ be an element of a $k$-closure-laminar matroid $M$, and let $C_1$ and $C_2$ be distinct circuits in $M/e$ with $r_{M/e}(\cl_{M/e}(C_1)\cap\cl_{M/e}(C_2))\geq k$. We aim to show that $\cl_{M/e}(C_1)\subseteq\cl_{M/e}(C_2)$ or $\cl_{M/e}(C_2)\subseteq\cl_{M/e}(C_1)$. This is certainly true if $r_M(C_1) = k$ or $r_M(C_2) = k$, so assume each of $|C_1|$ and $|C_2|$ is at least  $k+2$. As $k \in \{2,3\}$, it follows that $|C_i| \ge 2k-1$ for each $i$.

\begin{sublemma}
\label{subex}
For each $i$ in $\{1,2\}$, there is a circuit $D_i$ of $M$ such that $\cl_M(C_i \cup e) - \cl(\{e\})= \cl_M(D_i) - \cl(\{e\})$.
\end{sublemma}

To see this, first note that $C_i$ or $C_i \cup e$ is a circuit of $M$. In the latter case, we take $D_i = C_i\cup e$. In the former case, by Lemma~\ref{HamCir}, the result is immediate unless $\cl(C_i \cup e) = \cl(C_i) \cup \cl(\{e\})$, in which case we can take $D_i = C_i$.  Thus \ref{subex} holds. 

Now $r((\cl_{M}(C_1 \cup e)\cap\cl_{M}(C_2 \cup e)))\geq k+1$ as $r(\cl_{M/e}(C_1)\cap\cl_{M/e}(C_2))\geq k$. Hence, by \ref{subex}, 
$r(\cl_{M}(D_1)\cap\cl_{M}(D_2))\geq k$. Thus, for some $\{i,j\} = \{1,2\}$, we see that $\cl_M(D_i) \subseteq \cl_M(D_j)$. 
Hence
$\cl_M(C_i \cup e) - \cl_M(\{e\}) \subseteq \cl_M(C_j \cup e) - \cl_M(\{e\})$, so 
$\cl_{M/e}(C_i) - \cl_M(\{e\}) \subseteq \cl_{M/e}(C_j) - \cl_M(\{e\}).$ As each element of $\cl_M(\{e\}) - e$ is a loop in $M/e$, we deduce that $\cl_{M/e}(C_i)  \subseteq \cl_{M/e}(C_j).$ Thus the theorem holds.
\end{proof}

\begin{theorem}
\label{notk}
For all $k\geq 4$, the class of $k$-closure-laminar matroids is not minor-closed.
\end{theorem}

The proof of this theorem will use Bonin and de Mier's  characterization of matroids in terms of their collections of cyclic flats \cite[Theorem 3.2]{BDMCF}. 

\begin{theorem}
\label{bdm}
Let ${\mathscr Z}$
be a collection of subsets of a set $E$
and let $r$
be an integer-valued function on ${\mathscr Z}$. There is a matroid for which ${\mathscr Z}$
is the collection of cyclic flats and $r$
is the rank function restricted to the sets in ${\mathscr Z}$ if and only if
\begin{itemize}
\item[]{\rm (Z0)} ${\mathscr Z}$ is a lattice under inclusion;
\item[]{\rm (Z1)} $r(0_{\mathscr Z}) = 0;$
\item[]{\rm (Z2)}  $0 < r(Y) - r(X) < |Y - X|$ for all sets $X,Y$ in ${\mathscr Z}$ with $X \subsetneqq Y$; and
\item[]{\rm (Z3)} for all sets $X,Y$ in ${\mathscr Z}$,
$$r(X) + r(Y) \ge r(X\vee Y) + r(X \wedge Y) + |(X \cap Y) - (X \wedge Y)|.$$
\end{itemize}
\end{theorem}

\begin{proof}[Proof of Theorem~\ref{notk}.]
Let $A=\{a_1, a_2, \dots, a_{k-1}\}$, $B=\{b_1, b_2, \dots, b_{k-1}\}$, $C=\{c_1, c_2, \dots, c_{k-1}\}$, $D=\{e, a_1, b_1, c_1\}$, $C_a= A\cup C$, and $C_b= B\cup C$. Let $E = A \cup B \cup C \cup D$ and let ${\mathscr Z}$ be the following collection of subsets of $E$ having the specified ranks. 
\begin{center}
\begin{tabular}{| c | c|}
\hline
Rank $t$  &  Members of ${\mathscr Z}$ of rank $t$                       \\ \hline % size is
0      & $\emptyset$                                        \\ \hline % 0
$k$  & $C\triangle D$, $A\triangle D$, and $B\triangle D$ \\ \hline % k+2
$2k-3$ & $C_a$  and $C_b$                                   \\ \hline % 2k-2
$2k-2$ & $C_a\cup D$  and $C_b\cup D$                       \\ \hline % 2k
$2k-1$ & $E$                                           \\ \hline % 3k-2
\end{tabular}
\end{center}

We will show that ${\mathscr Z}$ is the collection of cyclic flats of a matroid $M$ on $E$. We then show that $M$ is $k$-closure-laminar  but that $M/e$ is not.

We note that $C\triangle D$, $A\triangle D$, $B\triangle D$, $C_a$, and $C_b$ form an antichain, that $C_a\cup D$ contains $C\triangle D$, $A\triangle D$, and $C_a$, but not $B\triangle D$ or $C_b$, and that $C_b\cup D$ contains $C\triangle D$, $B\triangle D$, and $C_b$, but not $A\triangle D$ or $C_a$.  This gives us that   ${\mathscr Z}$ is a lattice obeying (Z1) and which we can quickly check obeys (Z2). There are fifteen  incomparable pairs of members of  ${\mathscr Z}$, seven of which satisfy $r(X) + r(Y) = r(X\vee Y) + r(X \wedge Y) + |(X \cap Y) - (X \wedge Y)|$. For the remaining eight, the inequality $r(X) + r(Y) \ge r(X\vee Y) + r(X \wedge Y) + |(X \cap Y) - (X \wedge Y)|$ reduces to the inequality $k\geq 4$. Hence  ${\mathscr Z}$ obeys (Z3) for all $k\geq 4$, so $M$ is  a matroid. As noted in \cite{BDMCF}, its circuits are the minimal subsets $S$ of $E$ such that ${\mathscr Z}$ contains an element $Z$ containing $S$ with $|S| = r(Z) + 1$. 

To show that $M$ is $k$-closure-laminar, we note that $C_a\cup D$ is non-Hamiltonian for $2=|C_a\cup D|-r(C_a\cup D)$, yet there is no  element of $C_a\cup D$ that is in all three non-spanning circuits of $M|(C_a\cup D)$. By symmetry, $C_b\cup D$ is non-Hamiltonian. All of the other cyclic flats of $M$ are Hamiltonian.  This gives us ten pairs $(X,Y)$ of incomparable Hamiltonian flats for which to check that $r(X\cap Y)\leq k-1$. For all such pairs either $|X\cap Y|=2$ or $|X\cap Y|=k-1$, so $M$ is indeed $k$-closure-laminar. To see that $M/e$ is not $k$-closure-laminar, we note that $\cl_{M/e}(C_a)=C_a\cup b_1$ and $\cl_{M/e}(C_b)=C_b\cup a_1$. Then $\cl_{M/e}(C_a)\cap\cl_{M/e}(C_b)=C\cup \{a_1,b_1\}$, which has rank $k$ in $M/e$ as $(C \triangle D) - e$ is the only circuit of $M/e$ contained in it. Therefore $M/e$ is  not  $k$-closure-laminar as neither $\cl_{M/e}(C_a)$ nor $\cl_{M/e}(C_b)$ is contained in the other.
\end{proof}

\section{Excluded Minors}
\label{HGLM}

We now note some excluded minors for the classes of $k$-laminar and $k$-closure-laminar matroids. For $n\geq k+2$, let $M_n(k)$ be the truncation to rank $n$ of the cycle matroid of the graph consisting  of two vertices that are joined by three internally disjoint paths $P$, $X_1$, and $X_2$ of lengths $k$, $n-k$, and $n-k$, respectively. In particular,  $M_{4}(2) \cong M(K_{2,3})$. Observe that, when $k = 0$, the path $P$ has length $0$ so its endpoints are equal. Thus $M_n(0)$ is the truncation to rank $n$ of the direct sum of two $n$-circuits. 
Let $M^-(K_{2,3})$ be the unique matroid that is obtained by relaxing a circuit-hyperplane of $M(K_{2,3})$. For $n\geq k+3\geq 5$, let $N_n(k)$ be the truncation to rank $n$ of the graphic matroid that is obtained by attaching   two $(n-k)$-circuits to  distinct elements of a  $(k+2)$-circuit via parallel connection. For $n\geq k+2 \ge 4$, let $P_n(k)$ be the truncation to rank $n$ of the graphic matroid that is obtained by attaching   two $(n-k+1)$-circuits to  distinct elements of a  $(k+1)$-circuit via parallel connection. Thus $P_n(k)$ is a single-element contraction of $N_{n+1}(k)$. 
Moreover, $P_4(2)$ is isomorphic to the matroid that is obtained by deleting  a rim element from a rank-$4$ wheel. 

\begin{lemma}
\label{mnk}
For all $n\geq 4$, the matroid $M_n(k)$ is an excluded minor for the classes of $k$-laminar matroids and $k$-closure-laminar matroids.
\end{lemma}

\begin{proof}
We may assume that $k\geq2$, as the lemma holds for $k=0$ and for $k=1$ by results in \cite{Don} and \cite{OurPaper}.
Clearly $M_n(k)$ is  not $k$-laminar so is not $k$-closure-laminar.
If we delete an element of $M_n(k)$, then we get a matroid with at most one non-spanning circuit. By  Proposition~\ref{baby}(vi), such a matroid is   $k$-closure-laminar and hence is  $k$-laminar.  If we contract an element of $P$ from $M_n(k)$, we get a matroid that is $k$-closure-laminar since in it the closures of the  only two non-spanning circuits   meet in $k-1$ elements. Instead, if we contract an element of $X_1$ or $X_2$, we again get a matroid with exactly one non-spanning circuit. Thus the lemma holds. 
\end{proof}

Similar arguments give the  following result.

\begin{lemma}
\label{therest} $ $

\begin{itemize}
\item[(i)] The matroid $M^-(K_{2,3})$ is an excluded minor for the classes of $2$-laminar and $2$-closure-laminar matroids.
\item[(ii)] For all $n\ge k+3 \geq 5$, the matroid $N_n(k)$ is an excluded minor for the class of $k$-laminar matroids.
\item[(iii)] For all $n\ge k+2 \geq 4$, the matroid $P_n(k)$ is an excluded minor for the class of $k$-closure-laminar matroids.
\end{itemize}
\end{lemma}

The main result of this paper is that we have now identified all of the excluded minors for the classes of 2-laminar and 2-closure-laminar matroids. We will use the following basic results. 
We omit the elementary proof of the second one.

\begin{lemma}
\label{MatroidsThatBehaveLikeGraphsBehaveLikeGraphs}
Let $C$ be a circuit in  a matroid $M$ and let $x$ and $y$ be non-loop elements of $\cl(C)-C$. Suppose that $C_x$, $C_x'$, $C_y$, and $C_y'$ are circuits such that $C_x\cup C_x'=C\cup x$ and $C_y\cup C_y'=C\cup y$, where $C_x\cap C_x'=\{x\}$, and $C_y\cap C_y'=\{y\}$, while $C_x\triangle\{x,y\}=C_y$. Then $x$ and $y$ are parallel.
\end{lemma}

\begin{proof} Note that the hypotheses imply that $C_y$, $C_y'$, and $C$ are the only circuits contained in $C\cup y$, as any other such circuit must contain both $C-C_y$ and $C-C_y'$. Clearly  $C_x-x=C_y-y$ and   $C_x'-x=C_y'-y$.
Choose elements $e$ and $e'$ with $e\in C_x-x$ and $e'\in C_x'-x$. We use circuit elimination to get circuits $D$ and $D'$ with $D\subseteq (C_x\cup C_y)-e$ and $D'\subseteq (C_x'\cup C_y')-e'$. Of necessity, $\{x,y\}$ is contained in both $D$ and $D'$. We also note that $D\cap D'$ is contained in 
$$((C_x\cup C_y)-e)\cap((C'_x\cup C'_y)-e')=(C_x\triangle\{e,y\})\cap(C_x'\triangle\{e',y\})=\{x,y\}.$$ If $D\not=D'$, then elimination gives a circuit $D''\subseteq (D\cup D')-\{x\}\subseteq (C\cup y)-\{e,e'\}$. But then $D''$ is not equal to $C$, $C_y$, or $C_y'$, which are the only circuits contained in $C\cup y$. Thus $D=D'=\{x,y\}$.
\end{proof}

\begin{lemma}
\label{Obvious}
Let $C$ and $D$ be distinct circuits of a matroid $M$.
\begin{itemize}
\item[(i)] If $D\not\subseteq\cl(C)$, then $|D-\cl(C)| \ge 2$.
\item[(ii)] If $|D- C| = 1$ and $D'$ is a circuit contained in $C\cup D$ other than $C$ or $D$, then $D' \supseteq C - D$.
\end{itemize}
\end{lemma}

\begin{lemma}
\label{common}
Let $M$ be an excluded minor for $\Ms$ where $\Ms$ is the class of $2$-laminar or $2$-closure-laminar matroids. Let $C_1$ and $C_2$ be circuits of $M$ neither of which is contained in the closure of the other such that $|C_1 \cap C_2| \ge 2$ when $\Ms$ is the class of $2$-laminar matroids while $r(\cl(C_1) \cap \cl(C_2)) \ge 2$ otherwise. Then 
\begin{itemize}
\item[(i)] $E(M) = \cl(C_1) \cup \cl(C_2)$; 
\item[(ii)] $M$ has $\cl(C_1)$ and $\cl(C_2)$ as hyperplanes so $|C_1| = |C_2|$; and 
\item[(iii)] if $C$ is a circuit of $M$ that meets both $C_1 - \cl(C_2)$ and $C_2 - \cl(C_1)$, then  either  $C$ is spanning, or $C$ contains $C_ 1 \bigtriangleup C_2$. 
\end{itemize}
\end{lemma}

\begin{proof}
If $f \in E(M) - ( \cl(C_1) \cup \cl(C_2))$, then $C_i \subseteq \cl_{M\ba f}(C_j)$   for some $\{i,j\} = \{1,2\}$. Thus $C_i \subseteq \cl_{M}(C_j)$; a contradiction. Hence (i) holds. 
Certainly $C_2 - \cl(C_1)$ contains an element $e$. As $e \not\in \cl(C_1)$, if $\{x,y\}$ is an independent subset of $\cl(C_1) \cap \cl(C_2)$, then $\{x,y\}$ is independent in $M/e$. It follows without difficulty that $M/e \in \Ms$ so either $\cl_{M/e}(C_2 - e) \supseteq C_1$ or $\cl_{M/e}(C_1) \supseteq C_2-e$. The former yields a contradiction. Hence $\cl_M(C_1 \cup e) \supseteq C_2$, so $\cl_M(C_1 \cup e) = E(M)$. Thus $\cl(C_1)$ is a hyperplane of $M$. By symmetry, so is $\cl(C_2)$. Hence $|C_1| = |C_2|$, so (ii) holds.

Now let $C$ be a circuit of $M$ that meets both $C_1 - \cl(C_2)$ and $C_2 - \cl(C_1)$. 
As $C -\cl(C_2)$ is non-empty, $|C- \cl(C_2)| \ge 2$, so $|C \cap C_1| \ge 2$. Suppose $C$ is non-spanning. As $\cl(C_1)$ is a hyperplane and $C$ meets $C_2 - \cl(C_1)$, it follows that  
$\cl(C_1) \not \supseteq \cl(C)$ and $\cl(C) \not \supseteq \cl(C_1)$. Since $|C \cap C_1| \ge 2$, if $E(M) - (C \cup C_1)$ contains an element $e$, then, as $M\ba e \in {\Ms}$, we get a contradiction. Therefore $E(M) = C \cup C_1$. By symmetry, $E(M) = C \cup C_2$. Thus $C$ contains $C_ 1 \bigtriangleup C_2$, so (iii) holds.
\end{proof}

\begin{theorem}
\label{EM2LM}
The excluded minors for the class of $2$-laminar matroids are \linebreak $M^-(K_{2,3})$, $M_n(2)$ for all $n\geq 4$, and $N_n(2)$ for all $n\geq5$.
\end{theorem}

\begin{proof}
Suppose that $M$ is an excluded minor for the class of $2$-laminar matroids. Then $M$ has circuits $C_1$ and $C_2$ with $|C_1\cap C_2|\geq2$ such that neither $C_1$ nor $C_2$ is contained in the closure of the other. Moreover, $E(M)=C_1\cup C_2$ and $|C_1\cap C_2|=2$, otherwise we could delete an element of $E(M)=C_1\cup C_2$ or contract an element of $C_1\cap C_2$ and still get a non-2-laminar matroid.  Let $\{a,b\}=C_1\cap C_2$. Next we show the following.

\begin{sublemma}
\label{subsub}
If $g\in \cl(C_1)-C_1$, then there are circuits $G$ and $G'$ that meet in $\{g\}$ such that $G\cup G'=C_1\cup g$ and  
$\{a,b\} \subseteq G \subseteq \cl(C_2)$. Furthermore, $G$, $G'$, and $C_1$ are the only circuits contained in $C_1\cup g$ so $\cl(G) - C_2 = G - C_2$.
\end{sublemma}

As $g\in\cl(C_1)-C_1$, there are circuits $G$ and $G'$ with $g\in G\cap G'$ and $G\cup G'=C_1\cup g$. 
As $E(M) = C_1 \cup C_2$, we see  that $g\in C_2-C_1$. Since $C_2\not\subseteq \cl(C_1)$, there is an element $e$ in $C_2-\cl(C_1)$. Then $M\backslash e$ is 2-laminar. 
Suppose that both $G$ and $G'$ meet $\{a,b\}$. Then $|G\cap C_2|\geq 2$ and $|G'\cap C_2|\geq 2$.
Thus either $\cl(C_2)$ contains both $G$ and $G'$, or $C_2$ is contained in $\cl(G)$ or $\cl(G')$. In the former case, $C_1\subseteq\cl(C_2)$, while, in the latter case, $C_2\subseteq\cl(C_1)$. We deduce that we may assume that $\{a,b\}\subseteq G$ and neither $a$ nor $b$ is in $G'$. 
Now $C_2\cap G=\{a,b,g\}$. Then deleting  an element of $C_1-G$ shows that one of $G$ and $C_2$ is contained in the closure of the other. As $C_2\not\subseteq\cl(G)$, it follows that $G\subseteq \cl(C_2)$.

Since $C_1\not\subseteq\cl(C_2)$, we have at least two elements of $C_1$ that are not in $\cl(C_2)$ so $G'-G$ has at least two elements. Thus 
$\cl(G)\not\supseteq G'$.
By Lemma~\ref{Obvious}(ii), as  $G'$ avoids $\{a,b\}$, it follows that $a \not \in \cl(G')$. Hence $G\not\subseteq \cl(G')$.  
Since $e\in C_2-\cl(C_1)$ and $M\backslash e$ is 2-laminar, we deduce  that $G\cap G'=\{g\}$, and~\ref{subsub} holds.

As there are at least two elements in each of $C_2-\cl(C_1)$, $C_1-\cl(C_2)$, and $C_1\cap C_2$, it follows that $r(M)\geq 4$. Next we show

\begin{sublemma}
\label{l=1}
$|\cl(C_1)-C_1|=|\cl(C_2)-C_2|\leq1$.
\end{sublemma}

Suppose that $\cl(C_1)-C_1=\{g_1,g_2,\ldots, g_t\}$. For each $i$ in $\{1,2,\dots,t\}$, let $G_i$ and $G_i'$ be the associated circuits given by ~\ref{subsub} whose union is $C_1 \cup g_i$ where $\{a,b,g_i\}\subseteq G_i$ and  $G'_i =C_1\triangle G_i$. As $G_i$ and $G_j$ meet in at least two elements for distinct $i$ and $j$ in $\{1, 2, \ldots, t\}$, the closures of $G_1,G_2,\ldots, G_t$ form a chain under inclusion. Say  $\cl(G_1)\supseteq \cl(G_2)\supseteq\dots\supseteq\cl(G_t)$. Since $\cl(G_i)-C_2=G_i-C_2$, it follows that  $G_1-C_2\supseteq G_2-C_2\supseteq\dots \supseteq G_t-C_2$. Now let $\{f_1,f_2,\ldots,f_s\}=\cl(C_2)-C_2$. For each $f_i$, there are circuits $F_i$ and $F_i'$ whose union is $C_2 \cup  f_i$ such that 
 $\{a,b\}\subseteq F_i$  and $F_i' = C_2 \triangle F_i$. Moreover, we may assume that $\cl(F_1) \supseteq \cl(F_i)$ for all $i$. 
 
By \ref{subsub}, for all $i$,
$$F_i-\{a,b,f_i\}\subseteq  \cl(C_1)-C_1= \{g_1,g_2,\dots,g_t\} \subseteq \cl(G_1).$$ Thus $F_i - f_i \subseteq \cl(G_1)$ so   $f_i\in\cl(G_1)$. Hence, by \ref{subsub},  $f_i\in G_1$. Moreover, as $F_i\subseteq\{f_i,a,b,g_1,g_2,\dots,g_t\}$ and $\{a,b,g_1,g_2,\dots,g_t\}\subseteq\cl(G_1)$, we see that $\cl(F_i) \subseteq \cl(G_1)$.  Since $\{f_1,f_2,\ldots,f_s\} \subseteq G_1$, we deduce that  $G_1 = \{g_1,a,b,f_1,f_2,\ldots,f_s\}$.  As $\cl(F_1) \subseteq \cl(G_1)$, it follows by symmetry that   $\cl(F_1)=\cl(G_1)$. Moreover, symmetry also gives that 
$F_1 = \{g_1,a,b,f_1,f_2,\ldots,f_s\}$. 
Since  $G_1$ and $F_1$ are both circuits spanning the same set, they have the same cardinality, so 
%$$|\{a, b, f_1, g_1, g_2, \ldots, g_l\}|=|F_1|=|G_2|=|\{a, b, g_1, f_1, f_2, \ldots, f_m\}|.$$ Thus, 
$t=s$; that is, 
$$|\cl(C_1)-C_1|=|\cl(C_2)-C_2|.$$ 
By Lemma~\ref{MatroidsThatBehaveLikeGraphsBehaveLikeGraphs}, since $\{g_1,g_2,\ldots,g_t\}$ is independent, we get that $G_i-g_i\not=G_j-g_j$ for distinct $i$ and $j$. Thus 
$$t+3=|G_1|>|G_2|>\dots > |G_l|\geq4$$
 where the last inequality follows because $G_l$ is not a proper subset of $C_2$.

Now suppose that $|G_2|=|G_1|-1$  where $(G_1-g_1)-(G_2-g_2)= \{f_i\}$. Choose $e\in C_1-\cl(C_2)$. As $f_i \in G'_2 - G'_1$, strong circuit elimination on $G_1'$ and $G_2'$, both of which contain $e$, yields a circuit $D$ containing $f_i$ and avoiding $e$. Since $D$ avoids $\{a,b\}$, it follows that $\{g_1,g_2\} \subseteq D$. 
As $e \not\in C_2 \cup D$, we deduce that $D \subseteq \cl(C_2)$, otherwise we obtain the contradiction that 
$C_2 \subseteq \cl(D) \subseteq \cl(C_1)$. 
But $(G'_2 - g_2) -(G'_1 - g_1) = \{f_i\}$, so 
$D \subseteq G_2'\cup g_1$, and $(G_2'\cup g_1)\cap\cl(C_2)=\{f_i,g_1,g_2\}$. As $D \subseteq \cl(C_2)$, it follows that $D \subseteq \{f_i,g_1,g_2\}$. 
This is a contradiction to \ref{subsub} because $D\not\in\{F_i,F_i'\}$. We deduce that $|G_2| \le |G_1|-2$. Thus $|G_2| \le t+3 - 2 = t + 1$. Hence $|G_t| \le 3$; a contradiction. 
We conclude that \ref{l=1} holds.

By Lemma~\ref{common}, $\cl(C_1)$ and $\cl(C_2)$ are hyperplanes of $M$, so $|C_1| = |C_2|$.
Suppose that $|\cl(C_1)-C_1| = 0$. Then, by \ref{l=1}, $|\cl(C_2)-C_2| = 0$. Thus every circuit of $M$ other than $C_1$ or $C_2$ must meet both $C_1 - C_2$ and $C_2 - C_1$. Assume $M$ has such a circuit $C$ that is non-spanning. Then,   by Lemma~\ref{common}(iii),    $C_1 \bigtriangleup C_2 \subseteq C$. As $|C_1 \cap C_2| = 2$ but $C$ is non-spanning, it follows that $C = C_1 \bigtriangleup C_2$. Thus $r(C) = r(M) - 1$, so $r(M) - 1   = |C_1| + |C_2| - 5$. But $r(M) - 1 = r(C_1) = |C_1| - 1$. Hence $|C_2| = 4$ and, by symmetry, $|C_1| = 4$.   It follows easily that $M \cong M(K_{2,3}) \cong M_4(2)$. Now suppose that every circuit other than $C_1$ or $C_2$ is spanning. Then, letting $|C_1| = n$, we see that $|C_2| = n$ and $r(M) = r(C_1) + 1 = n$. It follows that  $M \cong M(K_{2,3}^-)$ when $n = 4$, while $M \cong M_n(2)$ when $n \ge 5$.

By \ref{l=1}, we may now suppose that $\cl(C_1)-C_1 = \{g\}$. Then  $\cl(C_2)-C_2 = \{f\}$, say. By \ref{subsub}, $\{a,b,g,f\}$ is a circuit of $M$ as are both $G' = (C_1 - \{a,b,f\}) \cup \{g\}$ and $F' = (C_2 - \{a,b,g\}) \cup \{f\}$. All circuits of $M$ other than $C_1, C_2, \{a,b,g,f\}, G'$, and $F'$ must meet both $C_1 - \cl(C_2)$ and $C_2 - \cl(C_1)$. Hence, by Lemma~\ref{common}(iii), every such circuit is spanning  as  $C_1 \bigtriangleup C_2$ properly contains $G'$.  Again letting $|C_1| = n$, we see that $|C_2| = n$ and $r(M) =  n$. Thus  $M\cong N_{n}(2)$ for some $n \ge 5$.
\end{proof}

\begin{theorem}
\label{EM2LCM}
The excluded minors for the class of $2$-closure-laminar matroids are $M^-(K_{2,3})$, $M_n(2)$ for all $n\geq4$, and $P_n(2)$ for all $n\geq5$.
\end{theorem}
\begin{proof}
Let $M$ be an excluded minor for the class of 2-closure-laminar matroids. Clearly $M$ is simple.
 Now $M$ has two circuits $C_1$ and $C_2$ with $r(\cl(C_1)\cap\cl(C_2))\geq2$ such that neither is a subset of the closure of the other. By Lemma~\ref{common}(i),   $E(M)=\cl(C_1)\cup \cl(C_2)$. Moreover, $|C_1\cap C_2|\leq2$, otherwise we could   contract an element of $C_1\cap C_2$ and still have a  matroid that is not 2-closure-laminar. Next we observe that 
 
\begin{sublemma}
\label{newton}
$(\cl(C_1) \cup \cl(C_2)) - (C_1\cup C_2) \subseteq \cl(C_1) \cap \cl(C_2)$.
\end{sublemma}

Assume the contrary. If $e \in \cl(C_1) - (\cl(C_2) \cup C_1)$, then, as $M\ba e$ is  $2$-closure-laminar, $\cl(C_1) - e \subseteq \cl(C_2)$, so
$\cl(C_1)   \subseteq \cl(C_2)$; a contradiction.

We break the rest of the proof into three cases based on the size of $C_1\cap C_2$.
%We show that 

\begin{sublemma}
\label{NotEmpty}
$C_1\cap C_2\not=\emptyset$.
\end{sublemma}

Assume the contrary. Let $\{x,y\}$ be a  subset of $\cl(C_1)\cap\cl(C_2)$. To show \ref{NotEmpty}, we first establish that 
\begin{sublemma}
\label{NotEmpty2}
$\{x,y\} \not\subseteq C_2$. 
\end{sublemma}

Suppose $\{x,y\}\subseteq C_2$. As $M|{(C_1\cup\{x,y\})}$ is connected, there is a circuit $D_1$ with $\{x,y\}\subseteq D_1\subseteq C_1\cup\{x,y\}$. Then, for $c$ in $C_1-D_1$, the matroid $M\backslash c$ is 2-closure-laminar. Now $C_2\not\subseteq \cl(D_1)$ since $\cl(D_1)\subseteq\cl(C_1)$. Thus  $D_1\subseteq\cl(C_2)$, so $C_1\cap\cl(C_2)$ is non-empty. Choose an element $z$ in $C_1\cap\cl(C_2)$. 
Now $M$ has circuits $C_x$ and $C_x'$, with $x\in C_x\cap C_x'$, and $C_x\cup C_x'=C_1\cup x$. It also has circuits $C_y$, and $C_y'$ with $y\in C_y\cap C_y'$ and $C_y\cup C_y'=C_1\cup y$. We may assume that $z\in C_x\cap C_y$. 
Then $\{x,z\}\subseteq\cl(C_x)\cap\cl(C_2)$. As $C_1-(C_x\cup C_2)$ is non-empty, this implies that $C_x\subseteq\cl(C_2)$ since $C_2\not\subseteq\cl(C_x)$ because $\cl(C_x)\subseteq\cl(C_1)$. Similarly, $C_y\subseteq\cl(C_2)$.

Suppose $(C_x-x)\cap (C_x'-x)$ is non-empty and choose $e$ in $C_x\cap C_x'$. Then, as $\{e,x\} \subseteq C_x \cap C'_x$ and $y\not\in C_x\cup C_x'$, either $C_x\subseteq\cl(C_x')$ or $C_x'\subseteq \cl(C_x)$. In the latter case, $C_x'\subseteq\cl(C_x)\subseteq\cl(C_2)$, so $C_1\subseteq\cl(C_2)$; a contradiction. Thus $C_x\subseteq\cl(C_x')$. But then $C_1$ and $C_x'$ have the same rank, and hence the same size. Then $C_x'=C_1\triangle \{x,c\}$ for some $c\in C_1$. Now consider the 2-closure-laminar matroid $M\backslash c$. In it, $C_x$ and $C_2$ are circuits as $c\not\in C_2$. Then $r_{M\backslash c}(\cl_{M\backslash c}(C'_x)\cap\cl_{M\backslash c}(C_2))\geq2$ so $C'_x\subseteq\cl_{M\backslash c}(C_2)$ or $C_2\subseteq\cl_{M\backslash c}(C'_x)$. As $C_x \subseteq \cl_M(C_2)$ and 
$\cl_{M\backslash c}(C_x')\subseteq \cl_M(C_1)$, we obtain the contradiction that $C_1 \subseteq \cl(C_2)$ or $C_2 \subseteq \cl(C_1)$. We conclude that $C_x\cap C_x'=\{x\}$. Likewise $C_y\cap C_y'=\{y\}$.

If there is some element $f$ in $C_x\cap C_y'$, then, as $f\in C_x$, we have $f\in\cl(C_2)$. But then $\{f,y\}\subseteq\cl(C_y')\cap\cl(C_2)$ Thus either $C_y'\subseteq\cl(C_2)$ or $C_2\subseteq\cl(C_y')$. The former cannot occur as $C_y\subseteq\cl(C_2)$; nor can the latter as $\cl(C'_y) \subseteq \cl(C_1)$. Hence $C_x\cap C_y'=\emptyset$.  Likewise, $C_y\cap C_x'=\emptyset$.  But then $C_x\triangle\{x,y\}=C_y$ and $C_x'\triangle\{x,y\}=C_y'$. Hence, by Lemma~\ref{MatroidsThatBehaveLikeGraphsBehaveLikeGraphs}, $x$ and $y$ are parallel;  a contradiction. Thus \ref{NotEmpty2} holds. 

Next we suppose that $x\in C_2$ and $y\not\in C_2$. Choose a circuit $D$ with $\{x,y\}\subseteq D\subseteq C_2\cup y$. Then $D\subseteq\cl(C_1)$, since $C_2-(D\cup C_1)\not=\emptyset$ and  $\cl(D)\cap\cl(C_1)$ has rank at least two, while $\cl(D)\subseteq\cl(C_2)$. Now $(D \cap C_2) - x$ certainly contains some element $d$. Then $\{x,d\} \subseteq \cl(C_1) \cap \cl(C_2)$. Applying \ref{NotEmpty2} gives a contradiction.

We may now assume that $\{x,y\} \cap C_2 = \emptyset$. Let $D$ be a circuit   with $\{x,y\}\subseteq D\subseteq C_2\cup\{x,y\}$.  Then $D\subseteq\cl(C_1)$ as  $C_2-(D\cup C_1) \neq \emptyset$. By replacing  $y$ by an element of $D \cap C_2$, we revert to the case eliminated in the last paragraph. Hence \ref{NotEmpty} holds.

Now, we consider the case when $|C_1\cap C_2|=1$. Let $C_1\cap C_2 = \{x\}$ and choose $y$ in $(\cl(C_1)\cap\cl(C_2))- (C_1\cap C_2)$. Suppose $y\not\in C_1\cup C_2$. Then $M$ has circuits $D_1$ and $D_2$ containing $\{x,y\}$ and contained in $C_1\cup y$ and $C_2\cup y$, respectively. Without loss of generality, as $E(M)-(D_1\cup D_2)$ is non-empty, we may assume that $D_1\subseteq\cl(D_2)$. Since $|D_1|\geq 3$, there is an element $z$ of $D_1-\{x,y\}$. Then $z$ is in $\cl(D_2)$ and so is in $\cl(C_2)$. Thus $\{x,z\}\subseteq C_1 \cap\cl_{M\backslash y}(C_2)$ and we obtain a contradiction. It follows, by \ref{newton}, that $C_1 \cup C_2 = \cl(C_1) \cup \cl(C_2)  = E(M)$.

We may now assume that $y\in C_1\cap \cl(C_2)$. Then $M$ has a circuit $D$ such that $\{x,y\} \subseteq D \subseteq  C_2\cup y$. 
Clearly $D\subseteq\cl(C_2)$. To see that $D\subseteq\cl(C_1)$, we note that $C_2-(D\cup C_1)\not=\emptyset$, and $C_1 \not\subseteq \cl(D)$ as $D \subseteq\cl(C_2)$.
We now have $D- x \subseteq \cl_{M/x}(C_1 - x) \cap \cl_{M/x}(C_2 - x)$. Thus $r_{M/x}(D - x) \le 1$, otherwise, for some $\{i,j\} = \{1,2\}$, we have $\cl_{M/x}(C_i - x) \subseteq \cl_{M/x}(C_j - x)$; a contradiction. As $y \in D - x$, we see that $r_M(D) = 2$, so $D = \{x,y,y'\}$ for some $y'$. We deduce that 
\begin{sublemma}
\label{ddd}
$\{x,y,y'\}$ is the only circuit of $M|(C_2 \cup y)$ containing $\{x,y\}$.
\end{sublemma}

We show next that 
\begin{sublemma}
\label{ddt}
$(C_2 - \{x,y'\}) \cup y$ is the only circuit of $M|(C_2 \cup y)$  containing $y$ but not $x$. 
\end{sublemma}

By Lemma~\ref{Obvious}(ii), every circuit $D'$ of $M$ that contains $y$, avoids $x$, and is contained in $C_2 \cup y$ must contain $(C_2 - \{x,y'\}) \cup y$.    If $y' \in D'$, then $D' = (C_2  \cup y) - x$. Using $D'$ and $D$, we find a circuit $D''$ containing $x$ and contained in $(C_2 \cup y) - y'$. As $D''$ must also contain $y$, we see that $\{x,y\} \subseteq D''$ and we showed in \ref{ddd} that $M$ has no such circuit. We conclude that \ref{ddt} holds.

By \ref{ddt} and symmetry, $M$ has  $(C_1 - \{x,y\}) \cup y'$ as a  circuit, say $C'_1$. Let $C'_2$ be the circuit $(C_2 - \{x,y'\}) \cup y$. Next we note that 

\begin{sublemma}
\label{ddu}
$\cl(C_2) - C_2 = \{y\}$ and  $\cl(C_1) - C_1 = \{y'\}$.
\end{sublemma}

Assume there is an element $y_1$ in $\cl(C_2) - C_2 - y$.  Then $\{y,y_1\}$ is a subset of $\cl_{M/x}(C_1 - x) \cap \cl_{M/x}(C_2 - x)$  that is independent in $M/x$. Thus $C_i - x \subseteq \cl_{M/x}(C_j - x)$ for some $\{i,j\} = \{1,2\}$, so $C_i \subseteq \cl_{M}(C_j)$; a contradiction. It follows that $\cl(C_2) - C_2 = \{y\}$. By symmetry,  $\cl(C_1) - C_1 = \{y'\}$.

By Lemma~\ref{common}, $M$ has 
$\cl(C_1)$ and $\cl(C_2)$ as hyperplanes, so $|C_1| = |C_2|$. 
Let $C$ be a circuit of $M$ that is not $C_1$, $C_2$, $C_1'$, $C_2'$, or $D$. If $y \in C \subseteq C_2 \cup y$, then, by \ref{ddd} and \ref{ddt}, $C$ is $D$ or $C'_2$.  We deduce that $C$  meets both $C_1-\cl(C_2)$ and $C_2-\cl(C_1)$. Then, by Lemma~\ref{common}(iii), either $C$ is spanning, or 
$C$ contains $C_1 \bigtriangleup C_2$. But $|C_1 \cap C_2| = 1$ so $C$ is spanning. 
 We conclude that $C_1$, $C_2$, $C_1'$, $C_2'$, and $D$ are the only non-spanning circuits of $M$. Hence  $M \cong P_n(2)$ for some $n \ge 4$.

Finally,  suppose $|C_1\cap C_2|=2$. Then $M$ is not 2-laminar so it has as a minor one of the matroids identified in Theorem~\ref{EM2LM}. But $M$ cannot have a $N_n(2)$-minor for any $n \ge 5$ as this matroid has  $P_{n-1}(2)$, an excluded minor for the class of 2-closure-laminar matroids,  as a proper minor.
Thus $M$ has as a minor $M^-(K_{2,3})$ or $M_n(2)$ for some $n\geq 4$.  The result follows by Lemmas~\ref{mnk} and \ref{therest}. 
\end{proof}

Our methods for finding the excluded minors for the classes of $k$-laminar and $k$-closure-laminar matroids for $k=2$ do not seem to extend to  larger values of $k$.

\section{Intersections with other classes of matroids}
\label{BS}

We now discuss how the classes of $k$-closure-laminar and $k$-laminar matroids relate to some other well-known classes of matroids.  Finkelstein~\cite{LF} showed that all laminar matroids are gammoids, so they are representable over all sufficiently large fields \cite{pw,ip}. An immediate consequence of the following easy observation is that, for all $k\ge 2$,  if $M$ is  a $k$-closure-laminar matroid or a $k$-laminar matroid, then $M$  need not be  representable and hence $M$ need not be a gammoid.

\begin{proposition}
If $r(M)\leq k+1$, then $M$ is $k$-laminar and $k$-closure-laminar.
\end{proposition}

We  use the   next  lemma to describe the intersection of   the classes of 2-laminar and 2-closure-laminar matroids with the classes of binary and ternary matroids. 

\begin{lemma}
\label{nb}
The matroid $M^-(K_{2,3})$ is ternary and non-binary;  each of $M_n(2)$  and  $P_n(2)$ has a $U_{n,2n-3}$-minor; and  $N_n(2)$ has a $U_{n,2n-4}$-minor.
\end{lemma}

\begin{proof}
The first part follows because $M^-(K_{2,3})$   can be obtained from $U_{2,4}$ by adding elements in series to two  elements of the latter. Next   we note that we get $U_{n,2n-3}$ both from $M_n(2)$ by deleting an element of the path $P$ and   from $P_n(2)$ by deleting the basepoints of the parallel connections involved in   its construction.   Finally, deleting the basepoints of the parallel connections involved in producing $N_n(2)$ gives $U_{n,2n-4}$.
\end{proof}

The next two results follow without difficulty by combining the last lemma with Theorems~\ref{EM2LM} and \ref{EM2LCM} as the set of excluded minors for $\Ms \cap \Ns$ where $\Ms$ and $\N$ are minor-closed classes of matroids consists of the minor-minimal matroids that are excluded minors for $\Ms$ or $\Ns$ (see, for example, \cite[Lemma 14.5.1]{James}). Recall that $N_5(2)$ and $P_4(2)$ are the  matroids obtained by adjoining, via parallel connection, two triangles across distinct elements of a $4$-circuit and a triangle, respectively.

\begin{corollary}
\label{B2CL} A matroid $M$  is binary and $2$-laminar if and only if it has no minor isomorphic to $U_{2,4}$, $M(K_{2,3})$, or $N_5(2)$.
\end{corollary}

\begin{corollary}
\label{B2CL} A matroid $M$ is binary and $2$-closure-laminar if and only if it has no minor isomorphic to $U_{2,4}$, $M(K_{2,3})$, or $P_4(2)$.
\end{corollary}

Similarly, we find the  excluded minors for the classes of ternary $2$-laminar matroids and ternary $2$-closure-laminar matroids by noting that deleting an element from $F^*_7$  produces $M(K_{2,3})$, so $F_7^*$ is not $2$-laminar. 

\begin{corollary}
\label{T2LM} A matroid $M$  is ternary and $2$-laminar if and only if it has no minor isomorphic to $U_{2,5}$, $U_{3,5}$, $F_7$,  $M^-(K_{2,3})$, $M(K_{2,3})$, or $N_5(2)$.
\end{corollary}

\begin{corollary}
\label{T2CLM} A matroid $M$  is ternary and $2$-closure-laminar if and only if it has no minor isomorphic to $U_{2,5}$, $U_{3,5}$, $F_7$,  $M^-(K_{2,3})$, $M(K_{2,3})$, or $P_4(2)$.
\end{corollary}

Next we describe the  intersection of the class of graphic matroids with the classes of  2-laminar and 2-closure-laminar matroids  both constructively and via excluded minors.  

\begin{corollary}
\label{T2LM} A matroid $M$  is graphic and $2$-laminar if and only if it has no minor isomorphic to $U_{2,4}$, $M(K_{2,3})$, $F_7$,  $M^*(K_{3,3})$,  or $N_5(2)$.
\end{corollary}

\begin{corollary}
\label{T2CLM} A matroid $M$  is graphic and $2$-closure-laminar if and only if it has no minor isomorphic to $U_{2,4}$, $M(K_{2,3})$, $F_7$,   or $P_4(2)$.
\end{corollary}

\begin{lemma}
\label{blah}
Let $M$ be a simple, connected, graphic matroid. Then $M$ is $2$-laminar if and only if    $M$ is a coloop, $M$ is isomorphic to $M(K_4)$, or $M$ is the cycle matroid of a  graph consisting of a cycle with at most two chords  such that, when there are two chords, they are of the form $(u,v_1)$ and $(u,v_2)$ where $v_1$ is adjacent to $v_2$.
\end{lemma}

\begin{proof}
Clearly each of the specified matroids is $2$-laminar. Now 
let $G$ be a simple, 2-connected graph. Suppose first that $G$ is not    outerplanar. By a theorem of Chartrand and Harary~\cite{CH}, either $G$ is $K_4$ or $G$ has $K_{2,3}$ as a minor. In the latter case, $M(G)$ is not 2-laminar. Now suppose that $G$ is outerplanar. If $G$ has two chords that are not of the form $(u,v_1)$ and $(u,v_2)$ where $v_1$ is adjacent to $v_2$, then $M(G)$ has $N_4(2)$ as a minor, and so is not 2-laminar. \end{proof}

\begin{proposition}
Let $M$ be a simple, connected, graphic  matroid. Then $M$ is $2$-closure-laminar if and only if $M$ is a   coloop, $M$ is isomorphic to $M(K_4)$, or $M$ is the cycle matroid of a  cycle with at most one chord.
\end{proposition}

\begin{proof}
This  follows from Lemma~\ref{blah} by noting that  the cycle matroid of a cycle  with two chords  of the form $(u,v_1)$ and $(u,v_2)$ where $v_1$ is adjacent to $v_2$ has $P_4(2)$ as a minor. 
\end{proof}

We now show  that the intersections of the classes of $k$-laminar and $k$-closure-laminar matroids with the class of paving matroids coincide. 

\begin{theorem}
\label{pav1}
Let $M$ be a paving matroid, and $k$ be a non-negative integer. Then  $M$ is  $k$-laminar if and only if 
$M$ is   $k$-closure-laminar.
\end{theorem}

\begin{proof}
By Proposition~\ref{baby}(i), it suffices to prove that  if  $M$ is  not $k$-closure-laminar, then $M$ is not $k$-laminar.  We use the elementary observation that, since $M$ is paving, for every flat $F$, either $F=E(M)$, or $M|F$ is   uniform.
Suppose that   $C_1$ and $C_2$ are circuits of $M$ for which $r(\cl(C_1)\cap\cl(C_2))\geq k$  but neither $\cl(C_1)$ nor $\cl(C_2)$ is contained in the other. Then neither $\cl(C_1)$ nor $\cl(C_2)$ is spanning. Hence  both $M|\cl(C_1)$ and $M|\cl(C_2)$ are uniform. Let $X$ be a basis of $\cl(C_1)\cap\cl(C_2)$. Then $M$ has circuits $C_1'$ and $C_2'$ containing $X$ such that $\cl(C_i')=\cl(C_i)$ for each $i$.  Thus $M$ is not $k$-laminar. 
\end{proof} 

It is well known that the unique excluded minor for the class of paving matroids is $U_{0,1} \oplus U_{2,2}$. Using this, in conjunction with Theorems~\ref{EM2LM} and \ref{pav1}, it is not difficult to obtain the following. 

\begin{corollary}
\label{T2LP} The following are equivalent for a matroid $M$.
\begin{itemize}
\item[(i)] $M$ is  $2$-laminar and paving;
\item[(ii)] $M$ is  $2$-closure-laminar and paving;
\item[(iii)] $M$ has no minor in $\{U_{0,1} \oplus U_{2,2}, M^-(K_{2,3})\} \cup \{M_n(2):  n \ge 4\}$.
\end{itemize}
\end{corollary}

\end{document}